\documentclass[11pt]{article}
\usepackage[a4paper, total={6in, 8in}]{geometry}
\usepackage{amsfonts}
\usepackage{graphicx}
\usepackage{epstopdf}
\usepackage{algorithmic}
\usepackage{amssymb}
\usepackage{amsmath}
\usepackage{amsthm}
\usepackage{accents}
\usepackage{comment}

\usepackage{caption}
\usepackage{subcaption}

\usepackage[eulergreek]{sansmath}

\renewcommand{\vec}[1]{\mathbf{\boldsymbol{#1}}}
\newcommand{\xvec}[1]{\vec{\mathsf{#1}}}

\newcommand{\xit}[2]{\xvec{#1}^{(#2)}}

\newcommand{\dx}{\Delta x}
\newcommand{\dt}{\Delta t}

\newcommand{\kron}{\otimes}

\newcommand{\diag}[1]{\mathrm{diag}(#1)}

\DeclareMathOperator{\sech}{sech}

\newtheorem{theorem}{Theorem}

\newtheorem{proposition}{Proposition}

\title{Resolving Entropy Growth from Iterative Methods}
\author{Viktor Linders$^{\mbox{\tiny\rm 1}}$, 
Hendrik Ranocha$^{\mbox{\tiny\rm 2}}$,
Philipp Birken$^{\mbox{\tiny\rm 1}}$}
\date{}

\begin{document}
\maketitle
\baselineskip=0.9
\normalbaselineskip
\vspace{-3pt}
\begin{center}
{
\footnotesize\em $^{\mbox{\tiny\rm 1}}$Centre for mathematical sciences, Lund University, Lund, Sweden\\
$^{\mbox{\tiny\rm 2}}$Applied Mathematics, University of Hamburg, Hamburg, Germany\\ email: \, viktor.linders@math.lu.se \\ mail@ranocha.de \\ \qquad \, philipp.birken@na.lu.se
}
\end{center}

\begin{abstract}
\noindent We consider entropy conservative and dissipative discretizations of nonlinear conservation laws with implicit time discretizations and investigate the influence of iterative methods used to solve the arising nonlinear equations. We show that Newton's method can turn an entropy dissipative scheme into an anti-dissipative one, even when the iteration error is smaller than the time integration error. We explore several remedies, of which the most performant is a relaxation technique, originally designed to fix entropy errors in time integration methods. Thus, relaxation works well in consort with iterative solvers, provided that the iteration errors are on the order of the time integration method. To corroborate our findings, we consider Burgers' equation and nonlinear dispersive wave equations. We find that entropy conservation results in more accurate numerical solutions than non-conservative schemes, even when the tolerance is an order of magnitude larger.
\end{abstract}


%========================================================================

\section{Introduction} \label{sec:introduction}

For many partial differential equations (PDEs) in computational fluid dynamics (CFD), the notion of mathematical entropy plays an important role. It can be that entropy is preserved or that the solution satisfies an entropy inequality.
% In many cases, it is connected directly to the stability of a numerical method.
Alternatively, for nonlinear systems such as the Euler equations, an entropy inequality can be used to select a unique weak solution, at least in 1D.
Enforcing such an entropy inequality on the discrete level has, for certain equations, led to schemes that are provably
convergent to weak entropy solutions \cite{kroohl:99}.

In recent years it has been observed that the robustness of high-order discontinuous Galerkin (DG) discretizations for compressible turbulent flows, greatly benefits from discrete entropy stability; see \cite{gassner2021novel} and the references therein for an overview. Following the groundbreaking paper \cite{fiscar:13}, a space-time
version of the DG spectral element method (DG-SEM) with a specific choice of fluxes, was proven to be entropy-stable for systems of hyperbolic conservation laws \cite{fswdgc:19},
giving the first fully discrete high order scheme with this property. This method is inherently implicit.

Proofs of entropy stability for schemes involving implicit time integration assume that the arising systems of nonlinear equations are solved exactly. In practice, iterative solvers are used, which are terminated when a tolerance is reached. These in turn, have rarely been constructed with the intention of preserving physical invariants. Indeed, the impact of iterative solvers on the conservation of linear invariants was analyzed in \cite{birken2022conservation,linders2022locally}, and it was found that many iterative methods violate global or local conservation. A recent study of Jackaman and MacLachlan \cite{jackaman2022preconditioned} focuses on Krylov subspace methods for linear problems with quadratic invariants, and is related in spirit to this paper. Given the importance of entropy estimates in discrete schemes, there is a need to analyze the role played by iterative solvers in this context.

We consider PDEs posed in a single spatial dimension with periodic boundary conditions:
\begin{equation}
  u_t = {\cal L}(u), \quad x\in(x_\mathrm{min}, x_\mathrm{max}], \quad t\in[t_0,t_e].
\end{equation}
Here, ${\cal L}$ is a nonlinear differential operator. The cases considered are such that there is a convex nonlinear functional $\eta(u)$, here called the entropy, which is conserved. By splitting the spatial terms in certain ways, entropy conservative semi-discretizations are obtained with the aid of skew-symmetric difference operators.
%We use summation by parts (SBP) derivative operators \cite{fernandez2014review,svard2014review} in space to obtain entropy-preserving semidiscretizations.
%--- for periodic boundary conditions, these SBP operators are basically symmetric or skew-symmetric operators, depending on the order of the derivative. SHOULD BE MOVED SOMEWHERE ELSE
Next, we apply well-known implicit time integration methods that are able to either conserve or dissipate entropy, resulting in nonlinear algebraic systems with entropy bounded solutions.

Using convergent iterative solvers within implicit time integration, the iteration error (and thus the entropy error) can be made arbitrarily small by iterating long enough. However, such an approach is ill-advised since the number of iterations dictates the efficiency of the scheme. On the other hand, it is desirable to iterate long enough to prevent the iteration error from affecting the time integration error. The iteration error should thus be kept on the order of the local time integration error, but smaller in magnitude. The interested reader will find an overview of the use of iterative solvers in computational fluid dynamics in \cite{birken:21}.

In this setting, we analyze the entropy behavior of Newton's method, which forms the basis of many iterative solvers. As it turns out, Newton's method can generate undesired growth (or decay) of the entropy error. We consider Burgers' equation in Section~\ref{sec:burgers} and demonstrate that Newton's method can turn an entropy dissipative scheme into an anti-dissipative method when the tolerance is comparable to the error in the time step. A detailed analysis reveals that the source of the entropy error is the Jacobian of the spatial discretization. In Section~\ref{sec:strategies}, strategies are evaluated that recover the entropy conservation after each Newton iteration. It is seen that these strategies come with a large cost to the efficiency of the solver, since they significantly reduce the convergence rate.

To get the best of both worlds --- entropy conservation and fast convergence --- we introduce relaxation methods in Section~\ref{sec:relaxation}. Under mild assumptions, relaxation can be used to recover any convex entropy. We apply relaxation to nonlinear dispersive wave equations in Sections~\ref{sec:kdv} and \ref{sec:bbm}. Experiments reveal that entropy conservation leads to numerical solutions of considerably higher quality than non-conservative schemes, even when larger tolerances on the iterations are used. We thus conclude that relaxation works well in combination with iterative methods. Finally, we summarize the developments and discuss our conclusions in Section~\ref{sec:conclusion}.

\subsection{Software}

The numerical experiments discussed in this article are implemented in
Julia \cite{bezanson2017julia}. We use the Julia packages
SummationByPartsOperators.jl \cite{ranocha2021sbp},
ForwardDiff.jl \cite{revels2016forward}, and
Krylov.jl \cite{montoison2020krylov}. All source code required to reproduce the
numerical experiments is available in our repository
\cite{linders2023resolvingRepro}.

%%=============================================================%%

\section{Burgers' equation} \label{sec:burgers}

As a starting point, we consider Burgers' equation as a classical model for nonlinear conservation laws. It is well known how to design both entropy conservative and entropy dissipative discretizations for this problem. The purpose of this section is to demonstrate that Newton's method can destroy the entropic behavior of the discretization and cause undesired entropy growth (or decay). This example serves as an illustration of a general truth: Iterative methods may destroy the design principles upon which a discretization has been built.

\subsection{The continuous case}

Consider the inviscid Burgers' equation defined on a 1D periodic domain,
\begin{equation} \label{eq:Burgers_equation}
u_t + 6 u u_x = 0, \quad u(x,t=0) = u_0.
\end{equation}
The factor 6 has been added for consistency with the Korteweg-de Vries equation discussed in
Section~\ref{sec:kdv}. Multiplying \eqref{eq:Burgers_equation} by the solution $u$ and integrating in space results in the identity
\[
\frac{\mathrm{d}}{\mathrm{d}t} \frac{1}{2} \| u \|^2 = 0,
\]
where periodicity has been used to eliminate the boundary terms. Here, $\| \cdot \|$ denotes the $L^2$ norm. Evidently, the quantity $\eta(u) = \frac{1}{2} \| u \|^2$ is conserved. Throughout, we refer to $\eta$ as an \emph{entropy} for Burgers' equation \eqref{eq:Burgers_equation}.

\subsection{The semi-discrete case}

It is well known how to design discretizations of \eqref{eq:Burgers_equation}
that mimic the entropy conservation of the continuous problem. Let $\xvec{D}$
be a skew-symmetric matrix that approximates the spatial derivative operator.
In this section, we use classical fourth-order accurate central finite
differences with periodic boundaries. Consider the semi-discretization
\begin{equation} \label{eq:Burgers_semidiscrete}
\xvec{u}_t + 2(\xvec{D} \diag{\xvec{u}} \xvec{u} + \diag{\xvec{u}} \xvec{D} \xvec{u}) = \xvec{0}.
\end{equation}
Here and elsewhere a sans serif font is used to denote quantities evaluated on a discrete, uniform computational grid with grid spacing $\dx$. The formulation \eqref{eq:Burgers_semidiscrete} constitutes a so-called skew-symmetric split form \cite[eq. (6.40)]{richtmyer1967difference}. It arises from the identity $6 u u_x = 2 \bigl((u^2)_x + u u_x\bigr)$.

Multiplying \eqref{eq:Burgers_semidiscrete} by $\dx \xvec{u}^\top$ and using the skew-symmetry of $\xvec{D}$ yields
\begin{align*}
\frac{\mathrm{d}}{\mathrm{d}t} \frac{1}{2} \| \xvec{u} \|^2 &= -2 \dx (\xvec{u}^\top \xvec{D} \diag{\xvec{u}} \xvec{u} + \xvec{u}^\top \diag{\xvec{u}} \xvec{D} \xvec{u}) \\
&= -2 \dx (\xvec{u}^\top \xvec{D} \diag{\xvec{u}} \xvec{u} + \xvec{u}^\top \xvec{D}^\top \diag{\xvec{u}} \xvec{u}) \\
&= -2 \dx (\xvec{u}^\top (\xvec{D} + \xvec{D}^\top) \diag{\xvec{u}} \xvec{u} = \xvec{0}.
\end{align*}
The final equality follows from the skew-symmetry of $\xvec{D}$. Here we have defined the discrete norm $\| \xvec{u} \|^2 \equiv \dx \xvec{u}^\top \xvec{u}$ with a slight notational abuse. Evidently the semi-discrete entropy $\eta(\xvec{u}) = \frac{1}{2} \| \xvec{u} \|^2$ is conserved.

\subsection{The fully discrete case}

Since the entropy $\eta$ is a quadratic functional of the solution, a fully discrete scheme that conserves entropy is obtained by discretizing \eqref{eq:Burgers_semidiscrete} in time using the implicit midpoint rule. Further, any $B$-stable Runge-Kutta method will be entropy dissipative
\cite[Section~357]{butcher2016numerical}. Here, we consider both the (conservative) implicit midpoint rule and the (dissipative) fourth-order, three-stage Lobatto~IIIC method. The entropy analyses in this section and the next are for simplicity reserved for the midpoint rule. The analogous results for Lobatto~IIIC are found in Appendix \ref{app:Lobatto}. They heavily utilize the fact that Lobatto~IIIC is equivalent to a Summation-By-Parts (SBP) method in time \cite{ranocha2019some}; see
\cite{versbach2022theoretical,nordstrom2013summation,lundquist2014sbp,boom2015high,linders2020properties}
for further developments of this topic.

For a generic system of ordinary differential equations $\xvec{u}_t = \xvec{f}(\xvec{u})$, the midpoint rule can be expressed as a Runge-Kutta method as follows:
\begin{align*}
\xvec{U} &= \xvec{u}^n + \frac{\dt_n}{2} \xvec{f}(\xvec{U}) \\
\xvec{u}^{n+1} &= \xvec{u}^n + \dt_n \xvec{f}(\xvec{U}) \equiv 2 \xvec{U} - \xvec{u}^n.
\end{align*}
Here, $\xvec{U}$ denotes an intermediate stage used to define the numerical solution $\xvec{u}^{n+1} \approx \xvec{u}(t_{n+1})$ in terms of the solution at the previous time step; $\xvec{u}^n \approx \xvec{u}(t_n)$. It is understood that $\dt_n = t_{n+1} - t_n$.

Applied to the semi-discretization \eqref{eq:Burgers_semidiscrete}, the fully discrete scheme becomes, after slight rearrangement,
\begin{equation} \label{eq:Burgers_fully_discrete}
\begin{aligned}
\xvec{F}(\xvec{U}) &:= \xvec{U} - \xvec{u}^n + \dt_n (\xvec{D} \diag{\xvec{U}} \xvec{U} + \diag{\xvec{U}} \xvec{D} \xvec{U}) = \xvec{0}, \\
\xvec{u}^{n+1} &\phantom{:}= 2 \xvec{U} - \xvec{u}^n.
\end{aligned}
\end{equation}
From the second line in \eqref{eq:Burgers_fully_discrete}, the entropy $\eta(\xvec{u}^{n+1})$ is given by
\begin{equation} \label{eq:entropy_midpoint}
\eta(\xvec{u}^{n+1}) = \frac{\dx}{2} (2 \xvec{U} - \xvec{u}^n)^\top (2 \xvec{U} - \xvec{u}^n) = \eta(\xvec{u}^n) + 2 \dx (\xvec{U}^\top \xvec{U} - \xvec{U}^\top \xvec{u}^n).
\end{equation}
Left multiplication of the first line in \eqref{eq:Burgers_fully_discrete} by $\xvec{U}^\top$ reveals that
\begin{equation} \label{eq:entropy_midpoint_derivation}
\begin{aligned}
\xvec{U}^\top \xvec{U} - \xvec{U}^\top \xvec{u}^n &= -\dt_n \xvec{U}^\top \xvec{f}(\xvec{U}) \\
&= -\dt_n \xvec{U}^\top (\xvec{D} \diag{\xvec{U}} + \diag{\xvec{U}} \xvec{D}) \xvec{U} \\
&= - \dt_n \xvec{U}^\top ((\xvec{D} + \xvec{D}^\top) \diag{\xvec{U}}) \xvec{U} = 0,
\end{aligned}
\end{equation}
where the skew-symmetry of $\xvec{D}$ has been used in the same way as in the semi-discrete analysis. Consequently, $\eta(\xvec{u}^{n+1}) = \eta(\xvec{u}^n)$, hence entropy is conserved.

\subsection{Newton's method} \label{subsec:Newtons_method}

The analysis above assumes that the stage equations in the first line of \eqref{eq:Burgers_fully_discrete} are solved exactly. In practice this is infeasible. Instead, the stage vector $\xvec{U}$ is approximated using iterative methods. Here we consider Newton's method as an illustrative example. If Newton iterates are computed until the residual is sufficiently small, then entropy conservation is effectively retained. However, for efficiency reason it is desirable to terminate the iterates when the residual is smaller than some tolerance, chosen to reflect the overall expected accuracy of the scheme.
We will show that Newton's method is not entropy conservative from one iteration to the next. Consequently, schemes utilizing Newton's method (and other iterative methods) tend to lose the design principle upon which the discrete scheme is built, namely entropy conservation.

The iterates produced by Newton's method applied to the stage equation in \eqref{eq:Burgers_fully_discrete} are obtained as
\begin{equation} \label{eq:Newtons_method}
\begin{aligned}
\xvec{F}'(\xit{U}{k}) \Delta \xvec{U} + \xvec{F}(\xit{U}{k}) &= \xvec{0}, \\
\xit{U}{k+1} &= \xit{U}{k} + \Delta \xvec{U}, \quad k=0,1,\dots
\end{aligned}
\end{equation}
Here, $\xvec{F}'$ denotes the Jacobian of $\xvec{F}$ and can be explicitly evaluated; see \cite{chan2022efficient}:
\begin{equation} \label{eq:Burgers_Jacobian}
\xvec{F}'({\xit{U}{k}}) = \xvec{I} + \dt_n \left( \diag{\xit{U}{k}} \xvec{D} + \diag{\xvec{D} \xit{U}{k}} + 2 \xvec{D} \diag{\xit{U}{k}} \right).
\end{equation}
Here, $\xvec{I}$ is the identity matrix. Inserting this expression into \eqref{eq:Newtons_method}, explicitly writing out $\xvec{F}(\xit{U}{k})$ from \eqref{eq:Burgers_fully_discrete}, and utilizing the fact that $\Delta \xvec{U} = \xit{U}{k+1} - \xit{U}{k}$ leads to the following equation for $\xit{U}{k+1}$:
\begin{equation} \label{eq:Newton_Burgers}
\begin{aligned}
\xit{U}{k+1} - \xvec{u}^n &+ \dt_n \left( \diag{\xit{U}{k}} \xvec{D} \xit{U}{k+1} + \xvec{D} \diag{\xit{U}{k}} \xit{U}{k+1} \right) \\
&+ \dt_n \underbrace{\left[ \diag{\xvec{D} \xit{U}{k}} + \xvec{D} \diag{\xit{U}{k}} \right]}_{\xvec{M}} \Delta \xvec{U} = \xvec{0}.
\end{aligned}
\end{equation}
This equation replaces the stage equations in \eqref{eq:Burgers_fully_discrete} when $k+1$ Newton iterations are performed.

Suppose that Newton's method is terminated after $k+1$ iterations and that the solution is updated as $\xvec{u}^{n+1} = 2 \xit{U}{k+1} - \xvec{u}^n$.
As in \eqref{eq:entropy_midpoint}, the entropy is given by
\[
\eta(\xvec{u}^{n+1}) = \eta(\xvec{u}^n) + 2 \dx \left( (\xit{U}{k+1})^\top \xit{U}{k+1} - (\xit{U}{k+1})^\top \xvec{u}^n \right).
\]
Upon computing the parenthesized term, note that the first line of \eqref{eq:Newton_Burgers} is precisely $\xit{U}{k+1} - \xvec{u}^n + \dt_n \xvec{f}(\xit{U}{k})$. By the same derivation as in \eqref{eq:entropy_midpoint_derivation} it follows that $(\xit{U}{k})^\top \xvec{f}(\xit{U}{k}) = 0$. Consequently,
\[
(\xit{U}{k+1})^\top \xit{U}{k+1} - (\xit{U}{k+1})^\top \xvec{u}^n =- \dt_n (\xit{U}{k+1})^\top \xvec{M} \Delta \xvec{U},
\]
and hence
\[
\eta(\xvec{u}^{n+1}) = \eta(\xvec{u}^n) - 2 \dx \dt_n (\xit{U}{k+1})^\top \xvec{M} \Delta \xvec{U}.
\]
A slightly more elegant expression is obtained by replacing $(\xit{U}{k})^\top$ with $\Delta \xvec{U}^\top$. This can be done since the vector $\xit{U}{k}$ lies in the kernel of $\xvec{M}^\top$. To see this, note that
\begin{align*}
\xvec{M}^\top \xit{U}{k} &= \diag{\xvec{D} \xit{U}{k}} \xit{U}{k} + \diag{\xit{U}{k}} \xvec{D}^\top \xit{U}{k} \\
&= \diag{\xvec{D} \xit{U}{k}} \diag{\xit{U}{k}} \xvec{1} - \diag{\xit{U}{k}} \diag{\xvec{D} \xit{U}{k}} \xvec{1} = \xvec{0}.
\end{align*}
Here, the skew-symmetry of $\xvec{D}$ has been used together with the fact that any vector $\xvec{v}$ satisfies $\xvec{v} = \diag{\xvec{v}} \xvec{1}$, where $\xvec{1}$ is the vector of ones. The final equality follows by the commutativity of diagonal matrices.

We summarize these observations in the following:

% PROPOSITION: ENTROPY MIDPOINT NEWTON
\begin{proposition} \label{prop:entropy_midpoint_Newton}
Consider the discretization \eqref{eq:Burgers_fully_discrete} of Burgers' equation \eqref{eq:Burgers_equation} and apply $(k+1)$ iterations with Newton's method to the stage equations. The entropy of the resulting numerical solution satisfies
\begin{equation} \label{eq:Newton_Burgers_entropy}
\eta(\xvec{u}^{n+1}) = \eta(\xvec{u}^n) - 2 \dx \dt_n \Delta \xvec{U}^\top \xvec{M} \Delta \xvec{U}.
\end{equation}
Thus, the entropy error $\eta(\xvec{u}^{n+1}) - \eta(\xvec{u}^n)$ depends on the \emph{indefinite} quadratic form
\[
\Delta \xvec{U}^\top \xvec{M} \Delta \xvec{U} = \Delta \xvec{U}^\top \left[ \diag{\xvec{D} \xit{U}{k}} + \xvec{D} \diag{\xit{U}{k}} \right] \Delta \xvec{U}.
\]
Consequently, Newton's method may cause both entropy growth and decay.
\end{proposition}
% END OF PROPOSITION

With Lobatto~IIIC in place of the midpoint rule, the entropy error has an identical form, although $\Delta \xvec{U}$ and $\xvec{M}$ incorporate all intermediate stages in this case. Assuming that the iterates converge, $\Delta \xvec{U}$ will decrease until the entropy error is vanishingly small. However, in practical applications, we terminate the iterations at some tolerance matching the accuracy of the scheme. In such circumstances, the entropy error must be corrected by other means.

As an illustrative example we compute a single time step for Burgers' equation \eqref{eq:Burgers_equation}. The spatial domain is set to $(-10,10]$ and the initial condition is given by
\begin{equation} \label{eq:initial_data}
u_0 = \frac{c}{2} \sech^2 \left( \frac{\sqrt{c}}{2} x \right),
\end{equation}
where $c=2$. In the spatial discretization we set $\dx = 0.1$. Both Lobatto~IIIC and the midpoint rule are used in time with $\dt = 0.5$. The entropy $\eta(\xvec{u}^n)$ (with $n=1$) and residual $\| \xvec{F}(\xit{U}{k}) \|$ are shown in Fig. \ref{fig:Newton_Burgers_entropy_residual} when $k$ Newton iterations have been used to approximate the solution to the stage equation in \eqref{eq:Burgers_fully_discrete}. Throughout, $\xit{U}{0} = \xvec{u}^n$ is used as initial guess. This choice is known to conserve linear invariants \cite{birken2022conservation,linders2022locally}.

The entropy visibly grows when too few Newton iterations are used. As expected, it drops back to its correct value with further iterations, testifying to the indefiniteness of the quadratic form in \eqref{eq:Newton_Burgers_entropy}. The residual converges quadratically as expected for Newton's method.

These examples show that if Newton's method is applied with a tolerance around $10^{-3}$, then entropy growth is to be expected in these simulations. Note that entropy growth may occur even for the provably entropy dissipative discretization.

\begin{figure}
\centering
\begin{subfigure}[ht]{0.9\textwidth}
\centering
  \includegraphics[width=\textwidth]{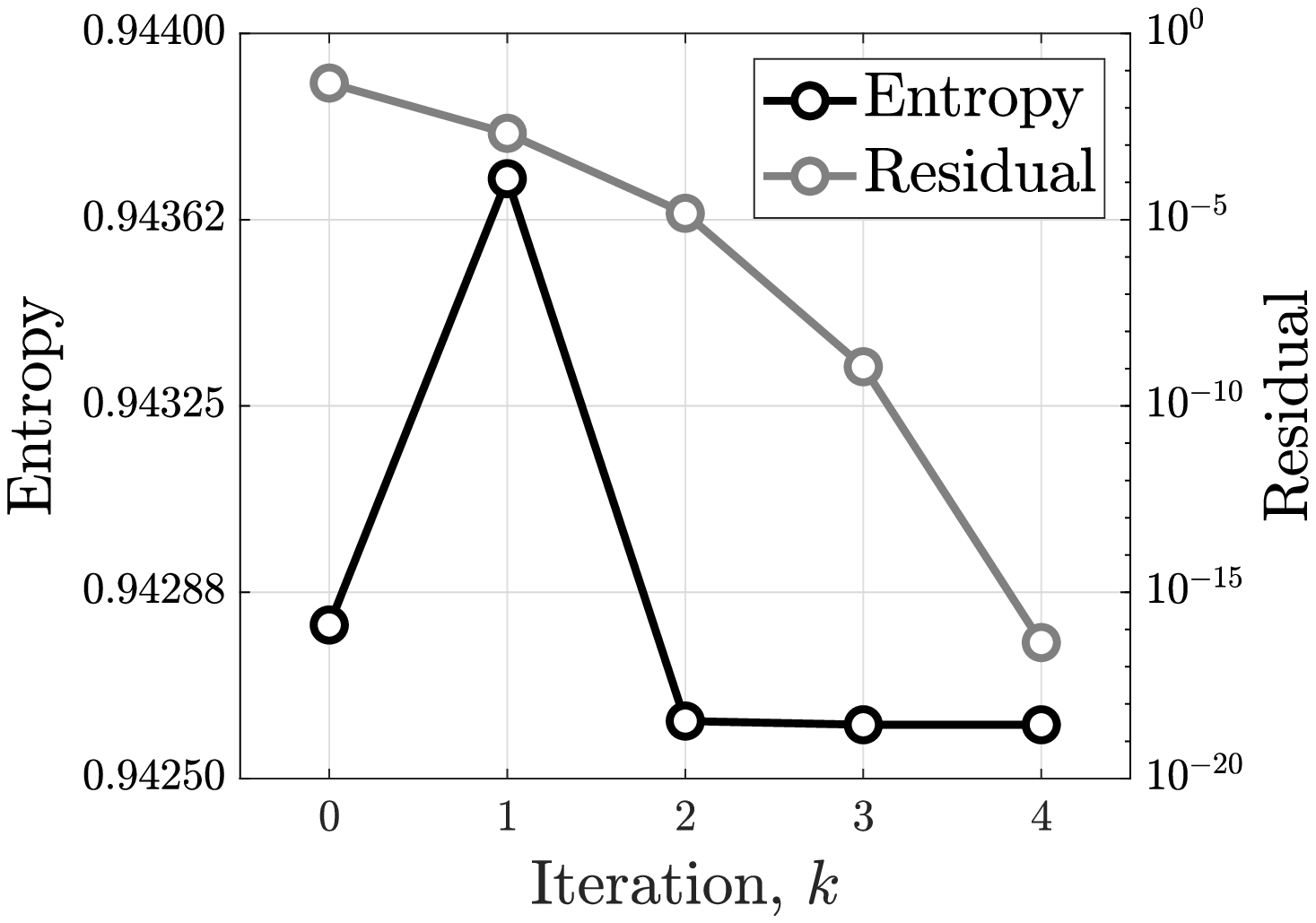}
  \caption{Entropy dissipative 3-stage Lobatto~IIIC.}
\end{subfigure}
\hfill
\begin{subfigure}[ht]{0.9\textwidth}
\centering
  \includegraphics[width=\textwidth]{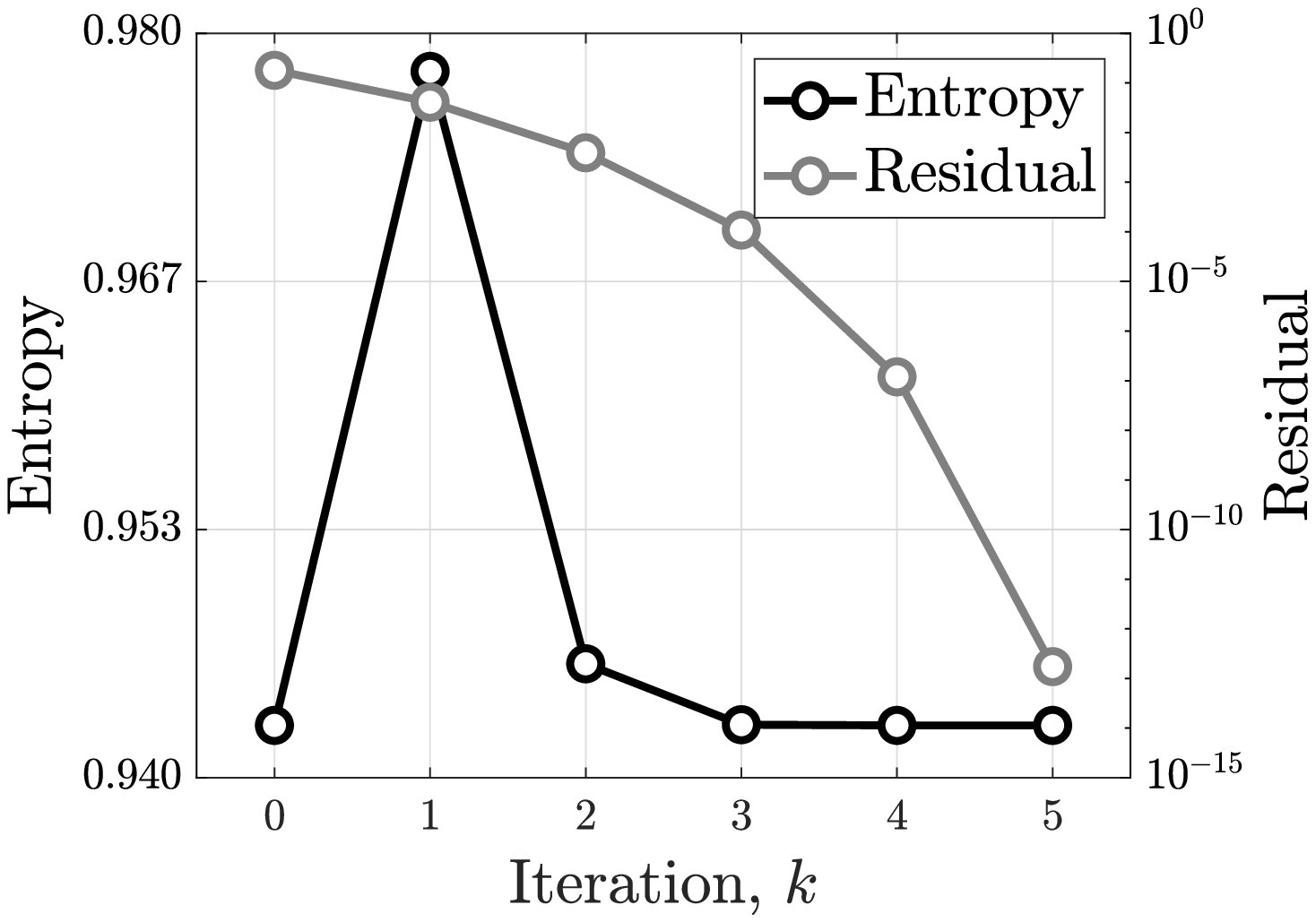}
  \caption{Entropy conservative implicit midpoint rule.}
\end{subfigure}
\caption{Entropy $\eta(\xvec{u}^n) = \| \xvec{u}^n\| / 2$ and residual
           $\| \xvec{F}(\xit{U}{k}) \|$ after $n=1$ time steps with
           $k$ Newton iterations applied to Burgers' equation with $\dx = 			   0.1$ and $\dt = 0.1$.}
\label{fig:Newton_Burgers_entropy_residual}
\end{figure}

%%=============================================================%%

\section{Strategies for recovering entropy conservation} \label{sec:strategies}

There are several ways in which we can recover entropy conservation within
Newton's method. The most obvious is to simply keep iterating until the entropy error is negligible. However, this process might be costly since we may be forced to iterate to residuals that are smaller than what is motivated by the accuracy of the discretization.

The entropy error in \eqref{eq:Newton_Burgers_entropy} opens up several avenues for modifying Newton's method to recover entropy conservation in each iteration. In the following subsections we describe two strategies that are independent of the choice of tolerance. However, as we will see, they both converge slower than Newton's method. The two methods are, respectively
\begin{itemize}
\item Method of Newton-type: By modifying the Jacobian, the entropy error is removed in its entirety.
\item Inexact Newton: Entropy conservation is recovered using a line search.
\end{itemize}
Throughout this section, we confine our attention to the entropy conservative discretization that utilizes the implicit midpoint rule.

%%=============================================================%%

\subsection{Method of Newton-type}

The entropy error \eqref{eq:Newton_Burgers_entropy} caused by Newton's method originates from the Jacobian of the spatial discretization. A simple remedy is thus to modify the Jacobian by removing the problematic terms, thereby obtaining a method of Newton-type. In other words, replacing the exact Jacobian $\xvec{F}'$ in \eqref{eq:Burgers_Jacobian} with the approximation
\[
\tilde{\xvec{F}}'({\xit{U}{k}}) := \xvec{I} + \dt_n \left( \diag{\xit{U}{k}} \xvec{D} + \xvec{D} \diag{\xit{U}{k}} \right)
\]
should result in a vanishing entropy error. However, this will simultaneously reduce the convergence speed from quadratic to linear \cite[Chapter 5]{kelley1995iterative}.

Repeating the experiment in section \ref{subsec:Newtons_method} reveals that this is indeed the case. Fig. \ref{fig:Newton_Burgers_entropy_residual_mod_Jacobian} shows the entropy and residual. Clearly the entropy is conserved as desired. However, the convergence is no longer quadratic but linear. After 14 iterations, the residual is comparable to four iterations with Newton's method. It is thus questionable if anything has been gained.

\begin{figure}[h]
\centering
%   \begin{tikzpicture}[
%     font=\normalsize
%     ]
%     \pgfplotsset{
%       cycle list name={modifiedcolorandmarkers},
%       xlabel={Iteration $k$},
%       width=0.8\textwidth,
%       height=0.53\textwidth,
%       grid=major,
%       legend columns=1,
%       xmin = -0.2,
%       xmax = 14.2,
%     }
%     \begin{axis}[
%         ylabel={Entropy},
%         axis y line*=left,
%         ymin = 0.938,
%         ymax = 0.982,
%       ]
%       \addplot+ [very thick, solid, mark=*, color=mycolor1,]
%         table [x index=0, y index=2]{Pics/burgers__midpoint_mod_jacobian.txt};
%         \label{plot:burgers-midpoint-mod-jacobian-entropy}
%     \end{axis}
%
%     \begin{axis}[
%         ylabel={Residual},
%         axis y line*=right,
%         ymode=log,
%       ]
%       \addlegendimage{/pgfplots/refstyle=plot:burgers-midpoint-mod-jacobian-entropy}\addlegendentry{Entropy}
%
%       \addplot+ [very thick, dotted, mark=*, color=mycolor2,]
%         table [x index=0, y index=1]{Pics/burgers__midpoint_mod_jacobian.txt};
%         \addlegendentry{Residual}
%     \end{axis}
%   \end{tikzpicture}%
% KEEP THE CODE ABOVE SINCE IT WAS USED TO PRODUCE THE PLOTS
  \includegraphics[width=0.8\textwidth]{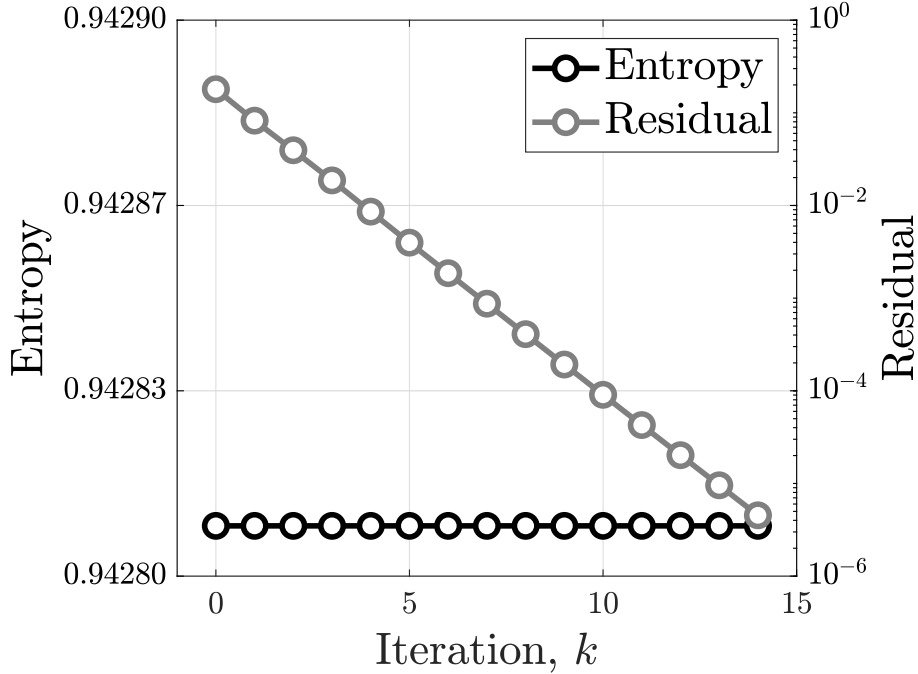}
  \caption{Entropy $\eta(\xvec{u}^n)$ and residual $\| \xvec{F}(\xit{U}{k}) \|$
           after $n=1$ time steps with $k \in \{ 0,\dots,14 \}$ iterations of method of Newton-type for the implicit midpoint rule applied to Burgers'
           equation with $\dx = 0.1$ and $\dt = 0.5$. The convergence is linear.}
  \label{fig:Newton_Burgers_entropy_residual_mod_Jacobian}
\end{figure}

%%=============================================================%%

\subsection{Inexact Newton}

As a second alternative, entropy conservation can be recovered through a line search. We replace Newton's method as stated in \eqref{eq:Newtons_method} by
\begin{equation} \label{eq:inexact_Newton}
\begin{aligned}
\xvec{F}'(\xit{U}{k}) ( \tilde{\xvec{U}} - \xit{U}{k}) + \xvec{F}(\xit{U}{k}) &= \xvec{0}, \\
\xit{U}{k+1} &= \alpha_k \tilde{\xvec{U}} + (1-\alpha_k) \xit{U}{k},
\end{aligned}
\end{equation}
where $\alpha_k \in [0,1]$ is a sequence of scalar parameters. This formulation is equivalent to the inexact Newton's method
\begin{equation*}
\begin{aligned}
\xvec{F}'(\xit{U}{k}) \Delta \xvec{U} + \xvec{F}(\xit{U}{k}) &= (1-\alpha_k) \xvec{F}(\xit{U}{k}), \\
\xit{U}{k+1} &= \xit{U}{k} + \Delta \xvec{U}.
\end{aligned}
\end{equation*}
If the sequence $\{ \alpha_k \}$ is such that $(1-\alpha_k) = \mathcal{O}(\| \xvec{F}(\xit{U}{k}) \|)$, then quadratic convergence is retained under standard assumptions. Under the much less severe condition $0 \leq 1-\alpha_k < 1$, covergence is generally linear \cite{eisenstat1996choosing}.

The essential property used in the proof that the fully discrete scheme \eqref{eq:Burgers_fully_discrete} is entropy conservative is that the stage vector $\xvec{U}$ satisfies $\xvec{U}^\top \xvec{U} - \xvec{U}^\top \xvec{u}^n = 0$. Proposition \ref{prop:entropy_midpoint_Newton} shows that the same relation does not hold for the Newton iterations. To ensure entropy conservation for the inexact Newton method, $\alpha_k$ must therefore be chosen such that
\[
(\xit{U}{k+1})^\top \xit{U}{k+1} - (\xit{U}{k+1})^\top \xvec{u}^n = 0,
\]
where $\xit{U}{k+1}$ is given as in \eqref{eq:inexact_Newton}. This is a quadratic equation in $\alpha_k$. Assuming that $\xit{U}{k}$ has correct entropy, the two roots are $\alpha_k = 0$ and
\[
\alpha_k = -\frac{\langle \Delta \xvec{U}, \xit{U}{k} \rangle + \langle \tilde{\xvec{U}}, \xit{U}{k} - \xvec{u}^n \rangle}{\| \Delta \xvec{U} \|^2}.
\]
Here, $\langle \cdot, \cdot \rangle$ denotes the Euclidean inner product scaled by $\dx$. The trivial root $\alpha_k = 0$ results in $\xit{U}{k+1} = \xit{U}{k}$ and is of no use, hence only the second root needs to be considered.

Repeating the previous experiment with inexact Newton leads to the results displayed in Fig. \ref{fig:Newton_Burgers_entropy_residual_line_search}. Entropy is conserved as expected. The convergence is faster than it was with the method of Newton-type, however it is still linear. This happens since $\alpha_k \in [0,1]$ for each $k$ but does not approach unity, as is necessary for quadratic convergence. Instead, its value appears to plateau around $0.95$ in this particular experiment.

\begin{figure}[ht]
\centering
  \includegraphics[width=0.8\textwidth]{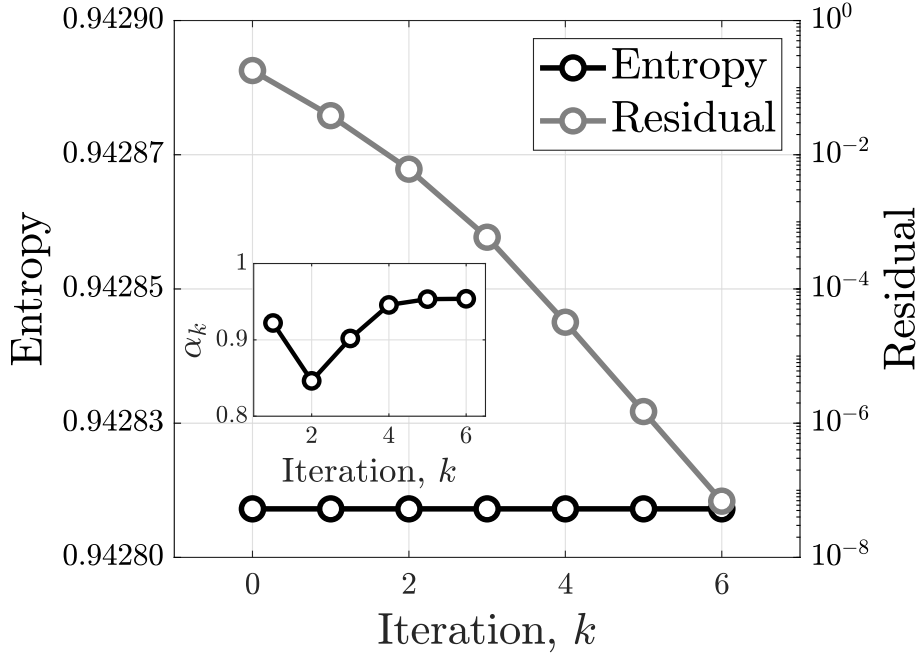}
  \caption{Entropy $\eta(\xvec{u}^n)$ and residual $\| \xvec{F}(\xit{U}{k}) \|$
           after $n=1$ time steps with $k \in \{ 0,\dots,6 \}$ inexact Newton
           iterations for the implicit midpoint rule applied to Burgers'
           equation with $\dx = 0.1$ and $\dt = 0.5$. The convergence is linear.}
  \label{fig:Newton_Burgers_entropy_residual_line_search}
\end{figure}

%%=============================================================%%

\section{Theory of relaxation methods} \label{sec:relaxation}

The idea of a line search discussed in the previous section can be applied, not necessarily after each Newton iteration, but alternatively just once after a full time step computed using multiple Newton iterations. The resulting schemes are known as relaxation methods
\cite{ketcheson2019relaxation,ranocha2020relaxation,ranocha2020general}. As we will see, relaxation solves the problem with entropy growth without impacting the convergence of Newton's method, or that of any other iterative method.

While a general theory of relaxation methods that includes multistep methods is available \cite{ranocha2020general}, we restrict the following discussion to one-step methods for simplicity. Thus, we consider a system of ODEs $\xvec{u}'(t) = f\bigl(\xvec{u}(t)\bigr)$ and a one-step method.
% computing $\xvec{u}^{n+1} \approx \xvec{u}(t^n + \dt)$ from the previous iterate $\xvec{u}^{n}$.
We again assume that there is a (nonlinear and sufficiently smooth) functional $\eta$ of interest, which we call an entropy.

In this setting, the basic idea of relaxation methods is to perform a line search along the secant line connecting the new approxiation $\xvec{u}^{n+1}$ produced by a given time integration method, and the previous value $\xvec{u}^{n}$. This results in the relaxed value
\begin{equation*}
  \xvec{u}^{n+1}_\gamma = \xvec{u}^{n} + \gamma (\xvec{u}^{n+1} - \xvec{u}^{n}),
\end{equation*}
where $\gamma \approx 1$ is the relaxation parameter chosen to enforce the
desired conservation or dissipation of $\eta$. This idea goes back to Sanz-Serna \cite{sanzserna1982explicit,sanzserna1983method} and Dekker \& Verwer \cite[pp. 265--266]{dekker1984stability}, who considered entropies $\eta$ given by inner product norms. However, the first approaches resulted in an order reduction \cite{calvo2006preservation}, which has been fixed in \cite{ketcheson2019relaxation} for inner product norms and extended to general entropies in \cite{ranocha2020relaxation,ranocha2020general}. Some further developments can be found in \cite{ranocha2020relaxationHamiltonian,ranocha2020fully,bencomo2023discrete,kang2022entropy,li2022implicit,li2023linearly}.

To motivate the approach, we augment the ODE by the equations $t' = 1$ and the entropy evolution given by the chain rule, leading to
\begin{equation*}
  \frac{\mathrm{d}}{\mathrm{d}t}
  \begin{pmatrix}
    t \\
    \xvec{u}(t) \\
    \eta\bigl(\xvec{u}(t)\bigr)
  \end{pmatrix}
  =
  \begin{pmatrix}
    1 \\
    f\bigl(\xvec{u}(t)\bigr) \\
    \eta'\bigl(\xvec{u}(t)\bigr) f\bigl(\xvec{u}(t)\bigr)
  \end{pmatrix}.
\end{equation*}
Given a suitable estimate of the entropy $\eta^\mathrm{new} = \eta\bigl(\xvec{u}(t^{n+1})\bigr) + \mathcal{O}(\dt^{p+1})$, where $p$ is the order of the time
integration method, relaxation methods enforce
\begin{equation} \label{eq:relaxation}
  \begin{pmatrix}
    t^{n+1}_\gamma \\
    \xvec{u}^{n+1}_\gamma \\
    \eta\bigl(\xvec{u}^{n+1}_\gamma\bigr)
  \end{pmatrix}
  =
  \begin{pmatrix}
    t^{n} \\
    \xvec{u}^{n} \\
    \eta\bigl(\xvec{u}^{n}\bigr)
  \end{pmatrix}
  + \gamma \dt
  \begin{pmatrix}
    1 \\
    \xvec{u}^{n+1} - \xvec{u}^{n} \\
    \eta^\mathrm{new} - \eta\bigl(\xvec{u}^{n}\bigr)
  \end{pmatrix}
\end{equation}
by inserting the second equation into the third one, resulting in a scalar
equation for the relaxation parameter
$\gamma$. Next, the numerical solution is updated according to the second
equation and the current simulation time is set to $t^{n+1}_\gamma$.
The last step is required to avoid order reduction. Finally,
the time integration is continued using $(t^{n+1}_\gamma, \xvec{u}^{n+1}_\gamma)$
instead of $(t^{n+1}, \xvec{u}^{n+1})$.

If entropy conservation is to be enforced, the canonical choice of the entropy estimate is $\eta^\mathrm{new} = \eta\bigl(\xvec{u}^{n}\bigr)$. In case of entropy dissipation, $\eta^\mathrm{new}$ can be estimated \cite{ketcheson2019relaxation,ranocha2020relaxation,ranocha2020general}: For an RK method
\begin{equation*}
\begin{aligned}
  \xvec{y}^i
  &=
  \xvec{u}^{n} + \dt \sum_{j=1}^{s} a_{ij} \, f\bigl(\xvec{y}^j\bigr),
  \qquad i \in \{1, \dots, s\},
  \\
  \xvec{u}^{n+1}
  &=
  \xvec{u}^{n} + \dt \sum_{i=1}^{s} b_{i} \, f\bigl(\xvec{y}^i\bigr),
\end{aligned}
\end{equation*}
with non-negative weights $b_i  \ge 0$, choose
\begin{equation*}
\label{eq:eta-estimate-RK}
  \eta^\mathrm{new}
  =
  \eta(\xvec{u}^{n}) + \dt \sum_{i=1}^s b_i (\eta' f)(\xvec{y}^{i}).
\end{equation*}
This leads to an entropy dissipative scheme.

The general theory of relaxation methods yields the following result
\cite{ranocha2020relaxation,ranocha2020general}:

% THEOREM: RELAXATION
\begin{theorem} \label{thm:relaxation}
  Consider the system of ODEs $\xvec{u}'(t) = f\bigl(\xvec{u}(t)\bigr)$.
  Assume $\xvec{u}^{n} = \xvec{u}(t^{n})$ and
  $\xvec{u}^{n+1} = \xvec{u}(t^{n+1}) + \mathcal{O}(\dt^{p+1})$
  with $p \ge 2$ and $t^{n+1} = t^{n} + \dt$. If
  $\eta^\mathrm{new} = \eta\bigl(\xvec{u}(t^{n+1})\bigr) + \mathcal{O}(\dt^{p+1})$,
  $\dt > 0$ is small enough, and the non-degeneracy condition
  \begin{equation*}
    \eta'(\xvec{u}^{n+1}) \frac{\xvec{u}^{n+1} - \xvec{u}^{n}}{\| \xvec{u}^{n+1} - \xvec{u}^{n} \|}
    =
    c \dt + \mathcal{O}( \dt^{2} ),
  \end{equation*}
  is satisfied with $c \neq 0$, then there is a unique $\gamma = 1 + \mathcal{O}( \dt^{p-1} )$
  that satisfies the relaxation condition \eqref{eq:relaxation} and the
  resulting relaxation method is of order $p$ such that
  \[
  \xvec{u}^{n+1}_\gamma = \xvec{u}(t^{n+1}_\gamma) + \mathcal{O}(\dt^{p+1}).
  \]
\end{theorem}
% END OF THEOREM

%For entropy conservative problems with $\eta^\mathrm{new} = \eta\bigl(\xvec{u}^{n}\bigr)$, relaxation methods are obviously entropy conservative, that is $\eta\bigl(\xvec{u}^{n+1}_\gamma\bigr) = \eta\bigl(\xvec{u}^{n}\bigr)$. For entropy dissipative problems with $\eta' f \le 0$ and the choice \eqref{eq:eta-estimate-RK} for a Runge-Kutta method with non-negative weights $b_i \ge 0$, relaxation methods are entropy dissipative; $\eta\bigl(\xvec{u}^{n+1}_\gamma\bigr) \le \eta\bigl(\xvec{u}^{n}\bigr)$.

The crucial observation for our application in this article is that there is
no further assumption on the preliminary time step update $\xvec{u}^{n+1}$
other than the accuracy constraint
$\xvec{u}^{n+1} = \xvec{u}(t^{n+1}) + \mathcal{O}(\dt^{p+1})$
in Theorem~\ref{thm:relaxation}. In particular,
the results can be applied to implicit Runge-Kutta methods where the stage
equations are solved inexactly with Newton's method (or another iterative method) up to some tolerance as long as the tolerance of the nonlinear solver does not affect the accuracy of the time integration method. A perturbation analysis shows that an approximation $\xvec{u}^{n+1} + \xvec{\varepsilon}$ results in a perturbed relaxation parameter that can be estimated. In general, we need that the size of the perturbation $\xvec{\varepsilon}$ is at most of the same order as the local error of the baseline time integration method, i.e. $\mathcal{O}(\dt^{p+1})$. This is precisely the situation desired for iterative methods in practice.

%%=============================================================%%

\section{Korteweg-de Vries equation} \label{sec:kdv}

We have motivated our study by looking at the influence of inexact solutions
obtained by Newton's method applied to provably entropy conservative and
entropy dissipative time discretization methods for Burgers' equation. Next,
we will look at a more complicated equation where
i) implicit methods are required due to stiffness constraints and
ii) entropy conservative methods have a significant advantage.
Indeed, the classical Korteweg-de Vries (KdV) equation
\begin{equation} \label{eq:kdv}
u_t + 6 u u_x + u_{xxx} = 0, \quad u(x,t=0) = u_0.
\end{equation}
is stiff due to the linear dispersive term $u_{xxx}$. We consider solutions of the form
\begin{equation} \label{eq:kdv-soliton}
u(x,t) = \frac{c}{2} \sech^2 \left( \frac{\sqrt{c}}{2} (x - ct) \right).
\end{equation}
In the context of the KdV equation, \eqref{eq:kdv-soliton} describes a soliton propagating with speed $c$. The behavior of numerical time integration methods applied to soliton solutions has been analyzed in \cite{frutos1997accuracy}: If the time integration method conserves the linear invariant $\int u$ and the quadratic invariant $\int u^2$, the error of the numerical solution has a leading-order term that grows linearly with time. Otherwise, the error grows quadratically in time at leading order.

De Frutos and Sanz-Serna \cite{frutos1997accuracy} verified this numerically using the implicit midpoint rule and an entropy non-conservative third order SDIRK method. They used essentially exact solutions of the stage equations. Under such conditions, relaxation has been demonstrated to yield the same reduced error growth when applied to the SDIRK method \cite{ranocha2020relaxationHamiltonian}.

Here, we extend the investigations and focus on the influence of
non-negligible tolerances of the nonlinear solver. In other words, we approximate the stage solutions to the stage equations with a tolerance that matches that of the time discretization.

\subsection{Numerical methods}

We consider the soliton solution \eqref{eq:kdv-soliton} of the KdV equation \eqref{eq:kdv} in the periodic domain $x \in (-10,10]$. The spatial discretization is the same as for Burgers' equation, i.e., five-point centered (i.e. skew-symmetric) differences with a spatial increment $\dx = 0.1$. In time, the implicit midpoint rule is used with time step $\dt = 0.05$. The final time is set to $t=1000$.

The solution to the nonlinear system arising within each time step is approximated with Newton-GMRES. Here, the absolute tolerance is set to zero and the relative tolerance is chosen as $tol \in \{ 10^{-3}, 10^{-4}, 10^{-5} \}$.

Applying relaxation to preserve a quadratic invariant yields a linear
equation for the relaxation parameter that we solve analytically. For the implicit midpoint rule, it is given by
\[
\gamma = \frac{\| \xvec{u}^n \|^2 - \langle \xvec{U}, \xvec{u}^n \rangle}{\| \xvec{U} - \xvec{u}^n \|^2}.
\]

The tolerances for GMRES are chosen using the procedure suggested by Eisenstat and Walker \cite{eisenstat1996choosing}, which is designed to recover quadratic convergence without over-iterating the linear solver in the early Newton iterations. This procedure requires the choice of two scalar parameters, $\eta_{\mathrm{max}}$ and $\gamma$ (not to be confused with the relaxation parameter). Here, we follow the implementation described in \cite[Chapter 6.3]{kelley1995iterative} with the parameter choices $(\gamma,\eta_{\mathrm{max}}) = (0.9,0.9)$.

\subsection{Numerical results}

Fig. \ref{fig:kdv} shows the discrete $L^2$ error and the error in the quadratic invariant $\eta(\xvec{u}) = \frac{1}{2} \|\xvec{u}\|^2$ for four cases: The unrelaxed solver with relative tolerance chosen from the set $\{ 10^{-3}, 10^{-4}, 10^{-5} \}$ and the relaxed solver with relative tolerance $10^{-3}$. The absolute tolerance is set to zero.

\begin{figure}[ht]
\centering
\leftskip -0.85cm
  \includegraphics[width=1.15\textwidth]{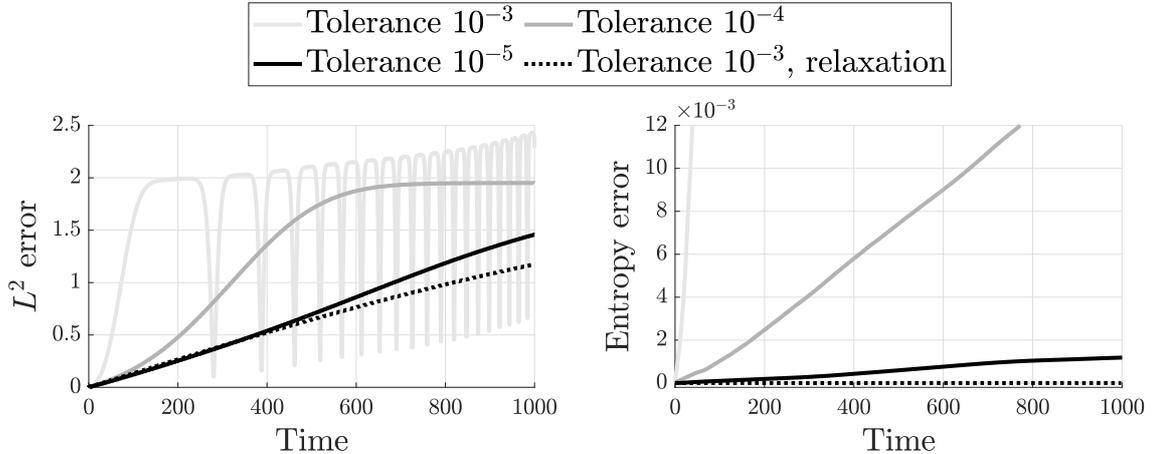}
  \caption{Error of numerical solutions and entropy error of the KdV equation \eqref{eq:kdv} with periodic boundary conditions using the implicit midpoint rule with time step size $\dt = 0.05$. The implicit equations are solved with Newton-GMRES and different relative tolerances.}
  \label{fig:kdv}
\end{figure}

The discrete $L^2$ error after the first time step is roughly $10^{-3}$.
Newton's method with a relative tolerance of $10^{-3}$ results in a superlinear error growth in time. The error eventually plateaus due to the fact that the numerical solution drifts completely out of phase with the true solution. It subsequently drifts in and out of phase, which gives rise to the oscillating nature of the error. Subsequent growth is also a consequence of the soliton losing its initial shape.

Shrinking the tolerance to $10^{-4}$ reduces the error growth rate significantly. However, the growth is still superlinear and reaches the same plateau seen with the larger tolerance. Yet another tolerance reduction to $10^{-5}$ leads to considerably improved results. The error grows linearly throughout the simulation time. However, the quadratic invariant is still not conserved.

By instead applying relaxation after every time step and using the original tolerance $10^{-3}$, the error growth rate is reduced even further. Additionally, the entropy is conserved up to machine precision. Thus, entropy conservation yields a better numerical solution even with a two orders of magnitude larger tolerance compared to the non-conservative scheme.

\begin{figure}[ht]
\centering
\leftskip -0.85cm
  \includegraphics[width=1.15\textwidth]{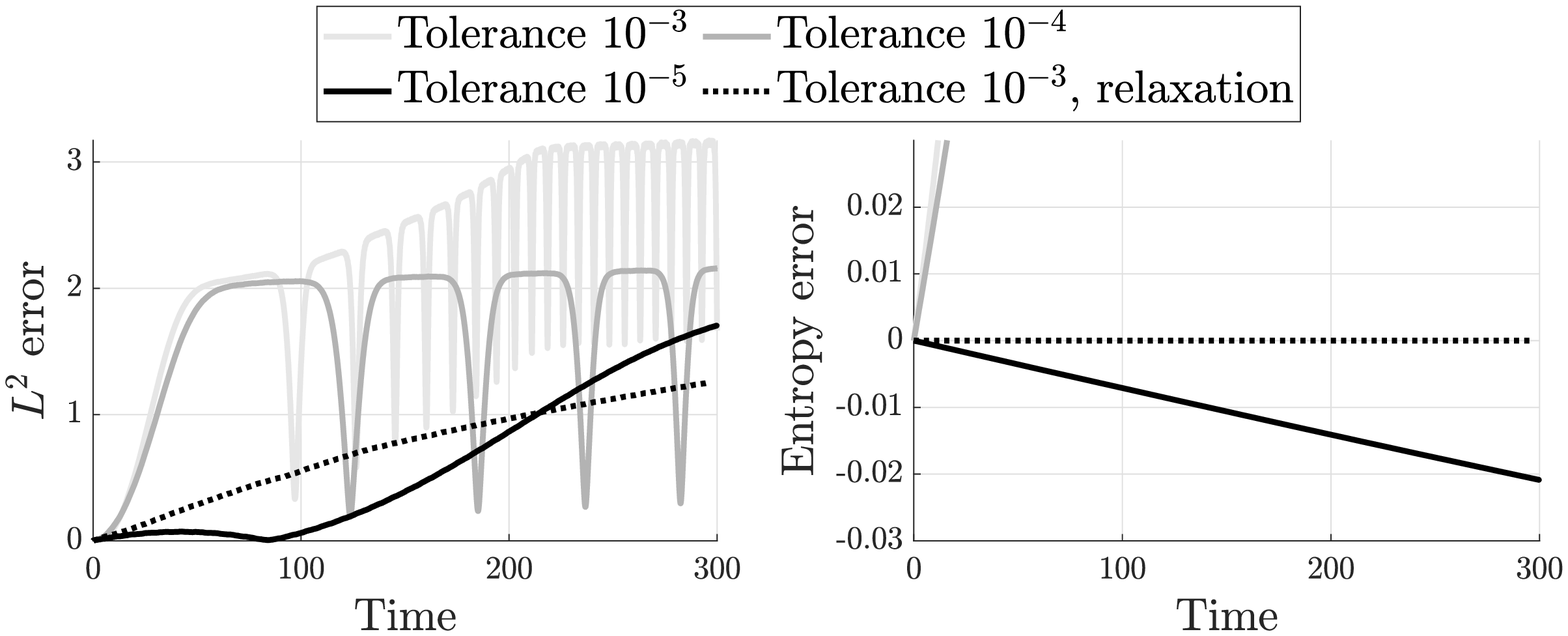}
  \caption{Error of numerical solutions and entropy error of the KdV equation \eqref{eq:kdv} with periodic boundary conditions using the fourth-order, three-stage Lobatto~IIIC method with time step size $\dt = 0.1$. The implicit equations are solved with Newton-GMRES and different absolute and relative tolerances.}
  \label{fig:kdv-lobatto}
\end{figure}

Figure~\ref{fig:kdv-lobatto} shows the results using the same space
discretization but the fourth-order, three-stage Lobatto~IIIC method in
time with $\dt = 0.1$. Since this method is $B$-stable, it dissipates the quadratic entropy when the stage equations are solved exactly. Here, we again apply Newton-GMRES with different tolerances. In this case, we set the absolute and relative tolerances to the same value, again chosen from the set $\{ 10^{-3}, 10^{-4}, 10^{-5} \}$. The $L^2$ error after one time step is approximately $10^{-3}$.

The resulting scheme is anti-dissipative with tolerances $10^{-3}$ and $10^{-4}$. The error displays similar characteristics to that of the midpoint rule with an initially superlinear growth.

With the tolerance $10^{-5}$ the entropy error is negative. The $L^2$ error shows a hump, suggesting that the numerical solution first drifts out of phase in one direction, then turns around and drifts the other way. Eventually the error starts to grow superlinearly as seen for the other tolerances.

As for the midpoint rule, relaxation leads to entropy conservation, a slower error growth and an overall better numerical solution. Again, this holds even when the tolerance is two orders of magnitude larger than the non-conservative scheme.

%%=============================================================%%

\section{Benjamin-Bona-Mahony equation} \label{sec:bbm}

An alternative to the stiff dispersive term of the KdV equation has been
proposed by Benjamin, Bona, and Mahony (BBM) in \cite{benjamin1972model},
leading to the BBM equation
\begin{equation} \label{eq:bbm}
  u_t + u_x + u u_x - u_{txx} = 0.
\end{equation}
Due to the mixed-derivative term $-u_{txx}$, the system is not stiff, but a linear elliptic equation must be solved to evaluate the time derivative $u_t$. Assuming periodic boundary conditions, the functionals
\begin{equation} \label{eq:bbm-invariants}
\begin{aligned}
  J^{\text{BBM}}_1(u)
  &= \int u,
  \\
  J^{\text{BBM}}_2(u)
  &= \frac{1}{2} \int ( u^2 + (u_x)^2 )
  = \frac{1}{2} \int u (\operatorname{I} - \partial_x^2) u,
  \\
  J^{\text{BBM}}_3(u)
  &= \int (u + 1)^3,
\end{aligned}
\end{equation}
are invariants of solutions to the BBM equation \cite{olver1979euler}. The error growth under time discretization has been analyzed in \cite{araujo2001error}: For methods conserving the linear invariant and one of the nonlinear invariants, the error of solitary waves grows linearly in time while general methods have a quadratically growing
error (both at leading order). We use this example to investigate the behavior of the methods also for non-quadratic entropies.

\subsection{Numerical methods}

We introduce numerical methods conserving important invariants of the
BBM equation \eqref{eq:bbm} with periodic boundary conditions. To broaden the scope of the following derivations, we assume that a discrete inner product is given by a diagonal, symmetric, and positive-definite matrix
$\xvec{M}$. For the classical finite-difference methods described above,
$\xvec{M} = \dx \xvec{I}$. Next, we need derivative operators $\xvec{D}_{1,2}$ approximating the first and second derivative operators. We require compatibility with integration by parts, i.e., we need that $\xvec{D}_1$ is skew-symmetric with respect to $\xvec{M}$ and that $\xvec{D}_2$ is symmetric and negative semidefinite with respect to $\xvec{M}$. This is satisfied by classical central finite differences and Fourier collocation methods but also by appropriate continuous and discontinuous Galerkin methods; see e.g. \cite{ranocha2021broad}.

For brevity, let $\xvec{u}^2 = \diag{\xvec{u}} \xvec{u}$ denote the elementwise square of $\xvec{u}$. The convention can similarly be extended to other elementwise powers as necessary.

Using such periodic derivative operators, the semidiscretization
\begin{equation} \label{eq:bbm-SBP-split}
  \xvec{u}_t
  + (\xvec{I} - \xvec{D}_2)^{-1} \left(
    \frac{1}{3} \xvec{D}_1 \xvec{u}^2
    + \frac{1}{3} \diag{\xvec{u}} \xvec{D}_1 \xvec{u}
    + \xvec{D}_1 \xvec{u}
  \right)
  =
  \xvec{0}
\end{equation}
conserves both the linear and the quadratic invariant \cite{ranocha2021broad}, which are discretely represented as
\begin{equation*}
  J^{\text{BBM}}_1(\xvec{u})
  =
  \xvec{1}^\top \xvec{M} \xvec{u},
  \qquad
  J^{\text{BBM}}_2(\xvec{u})
  =
  \xvec{u}^\top (\xvec{I} - \xvec{D}_2) \xvec{u}.
\end{equation*}
All symplectic Runge-Kutta methods, such as the implicit midpoint rule, conserve these linear and quadratic invariants
\cite{hairer2006geometric}.

Next, we construct a method conserving the cubic invariant.

% THEOREM: BBM-SBP-CENTRAL
\begin{theorem} \label{thm:bbm-SBP-central}
  The semidiscretization
  \begin{equation} \label{eq:bbm-SBP-central}
    \xvec{u}_t
    + (\xvec{I} - \xvec{D}_2)^{-1} \xvec{D}_1 \left(
      \frac{1}{2} \xvec{u}^2
      + \xvec{u}
    \right)
    =
    \xvec{0}
  \end{equation}
  conserves the linear invariant $J^{\text{BBM}}_1(\xvec{u})$ and the cubic invariant
  \begin{equation*}
    J^{\text{BBM}}_3(\xvec{u})
    =
    \xvec{1}^\top \xvec{M} (\xvec{1} + \xvec{u})^3
  \end{equation*}
  for commuting periodic derivative operators $\xvec{D}_1$, $\xvec{D}_2$.
\end{theorem}
% END OF THEOREM

% PROOF
\begin{proof}
  Conservation of the linear invariant follows from Lemma~2.1 and Lemma~2.2 of \cite{ranocha2021broad} since
  \begin{equation*}
    \xvec{1}^\top \xvec{M} \xvec{u}_t
    =
    - \xvec{1}^\top \xvec{M} (\xvec{I} - \xvec{D}_2)^{-1} \xvec{D}_1 \left(
      \frac{1}{2} \xvec{u}^2
      + \xvec{u}
    \right)
    =
    - \xvec{1}^\top \xvec{D}_1 \left(
      \frac{1}{2} \xvec{u}^2
      + \xvec{u}
    \right)
    =
    0.
  \end{equation*}
  Given conservation of the linear invariant, conservation of the cubic
  invariant is equivalent to conservation of the Hamiltonian
  \begin{equation*}
    \mathcal{H}(\xvec{u})
    =
    \xvec{1}^\top \xvec{M} \left(
      \frac{1}{6} \xvec{u}^3
      + \frac{1}{2} \xvec{u}^2
    \right).
  \end{equation*}
  This Hamiltonian is conserved, since $\xvec{M} (\xvec{I} - \xvec{D}_2)^{-1} \xvec{D}_1$ is skew-symmetric \cite[Lemma~2.3]{ranocha2021broad} and
  \begin{equation*}
    \mathcal{H}_t(\xvec{u})
    =
    \mathcal{H}'(\xvec{u}) \xvec{u}_t
    =
    -\left(
      \frac{1}{6} \xvec{u}^3
      + \frac{1}{2} \xvec{u}^2
    \right)^\top
    \xvec{M} (\xvec{I} - \xvec{D}_2)^{-1} \xvec{D}_1
    \left(
      \frac{1}{6} \xvec{u}^3
      + \frac{1}{2} \xvec{u}^2
    \right)
    =
    0.
  \end{equation*}
\end{proof}
% END OF PROOF

The average vector field (AVF) method \cite{mclachlan1999geometric} for an ODE
$\xvec{u}'(t) = f\bigl( \xvec{u}(t) \bigr)$ is given by
\begin{equation*}
  \xvec{u}^{n+1} = \xvec{u}^n + \dt \int_0^1 f\bigl( s \xvec{u}^{n+1} + (1 - s) \xvec{u}^n \bigr) \mathrm{d} s.
\end{equation*}
It conserves the Hamiltonian $\mathcal{H}$ of Hamiltonian systems
$\xvec{u}'(t) = S \, \mathcal{H}'\bigl( \xvec{u}(t) \bigr)$ with a constant (in time) skew-symmetric operator $S$ \cite{quispel2008new}. For quadratic Hamiltonians, it is equivalent to the implicit midpoint rule; for (up to) quartic Hamiltonians, it is equivalent to
\begin{equation} \label{eq:avf-simpsons-rule}
  \xvec{u}^{n+1} = \xvec{u}^n + \frac{\dt}{6} \left(
    f(\xvec{u}^{n})
    + 4 f\biggl(\frac{\xvec{u}^{n+1} + \xvec{u}^{n}}{2}\biggr)
    + f(\xvec{u}^{n+1})
  \right)
\end{equation}
see \cite{celledoni2009energy}. Thus, we obtain

% PROPOSITION
\begin{proposition}
  The time integration method \eqref{eq:avf-simpsons-rule} applied to the
  semidiscretization \eqref{eq:bbm-SBP-central} of the BBM equation with
  periodic boundary conditions conserves the linear and cubic invariants
  under the conditions given in Theorem~\ref{thm:bbm-SBP-central}.
\end{proposition}
% END OF PROPOSITION

\subsection{Numerical results}

We consider the traveling wave solution
\begin{equation}
\label{eq:bbm-traveling-wave}
  u(t,x) = A \sech\bigl( K (x - c t) \bigr)^2,
  \quad A = 3 (c - 1),
  \quad K = \frac{1}{2} \sqrt{1 - 1 / c},
\end{equation}
with speed $c = 1.2$ in the periodic domain $(-90, 90]$. We apply the
semidiscretizations described above with Fourier collocation methods \cite[Chapter 4]{kopriva2009implementing}
using $2^6$ nodes. The time integration methods use fixed time steps
$\dt = 0.25$. The nonlinear systems are solved with the same Newton-GMRES method as for the KdV equation. The relative tolerances are, as before, chosen as $tol \in \{ 10^{-3}, 10^{-4} \}$.

Applying relaxation to preserve a quadratic invariant yields a linear
equation for the relaxation parameter that we solve analytically. For
a cubic invariant, the relaxation parameter is determined by a quadratic
equation that we solve analytically; we always choose the solution closer
to unity.

\begin{figure}
\centering
	\begin{subfigure}{\textwidth}
    \includegraphics[width=\textwidth]{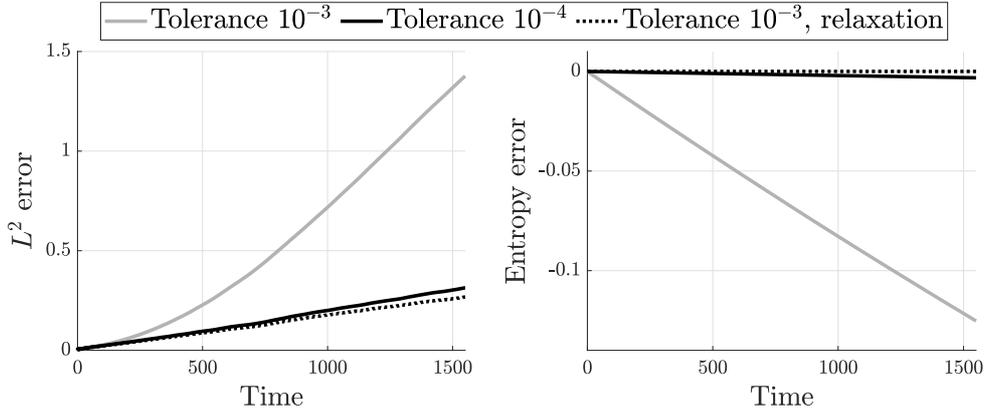}
    \caption{Discrete $L^2$ and entropy error for the quadratic invariant $J_2^{BBM}(\xvec{u})$. The implicit midpoint rule is applied to the semidiscretization \eqref{eq:bbm-SBP-split}.}
  	\end{subfigure}%

  	\begin{subfigure}{\textwidth}
  	\includegraphics[width=\textwidth]{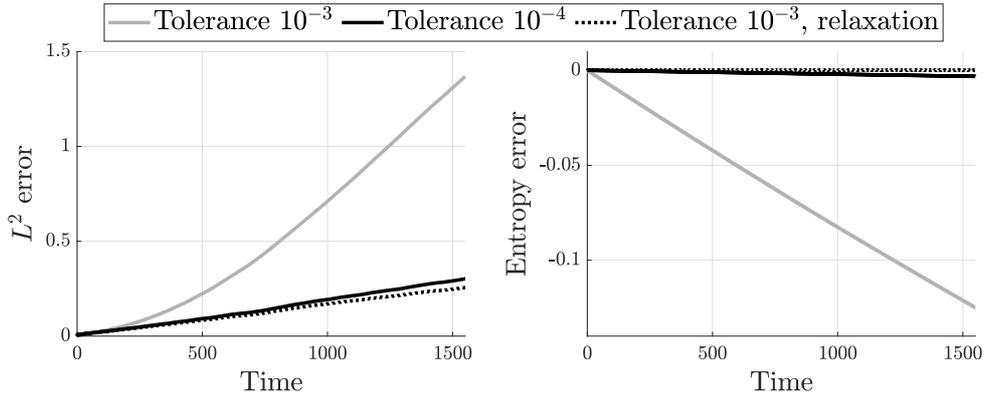}
    \caption{Discrete $L^2$ and entropy error for the cubic invariant $J_3^{BBM}(\xvec{u})$. The AVF method is applied to the semidiscretization \eqref{eq:bbm-SBP-central}.}
  	\end{subfigure}%
  	\caption{Discrete $L^2$ error of numerical methods for the BBM equation
           \eqref{eq:bbm} with periodic boundary conditions. The implicit
           equations are solved with Newton-GMRES and different relative
           tolerances.}
  \label{fig:bbm_error}
\end{figure}

\begin{comment}
\begin{figure}[ht]
\centering
\includegraphics[width=0.75\textwidth]{Pics/bbm__legend}
  \\
  \begin{subfigure}{0.49\textwidth}
    \includegraphics[width=\textwidth]{Pics/bbm_quadratic__error}
    \caption{Discrete $L^2$ error of the implicit midpoint rule
             applied to the semidiscretization \eqref{eq:bbm-SBP-split}.}
  \end{subfigure}%
  \hspace*{\fill}
  \begin{subfigure}{0.49\textwidth}
    \includegraphics[width=\textwidth]{Pics/bbm_quadratic__quadratic}
    \caption{Invariant $J_2^{BBM}(\xvec{u})$ of the implicit midpoint rule
             applied to the semidiscretization \eqref{eq:bbm-SBP-split}.}
  \end{subfigure}%
  \\
  \begin{subfigure}{0.49\textwidth}
    \includegraphics[width=\textwidth]{Pics/bbm_cubic__error}
    \caption{Discrete $L^2$ error of the AVF method \eqref{eq:avf-simpsons-rule}
             applied to the semidiscretization \eqref{eq:bbm-SBP-central}.}
  \end{subfigure}%
  \hspace*{\fill}
  \begin{subfigure}{0.49\textwidth}
    \includegraphics[width=\textwidth]{Pics/bbm_cubic__cubic}
    \caption{Invariant $J_3^{BBM}(\xvec{u})$ of the AVF method \eqref{eq:avf-simpsons-rule}
             applied to the semidiscretization \eqref{eq:bbm-SBP-central}.}
  \end{subfigure}%
  \caption{Discrete $L^2$ error of numerical methods for the BBM equation
           \eqref{eq:bbm} with periodic boundary conditions. The implicit
           equations are solved with Newton-GMRES and different relative
           tolerances.}
  \label{fig:bbm_error}
\end{figure}
\end{comment}

The results of these numerical experiments are shown in
Fig.~\ref{fig:bbm_error}. The discrete $L^2$ error after the first
time step is of the order $10^{-3}$. Using a relative tolerance of
$10^{-3}$ for Newton-GMRES results in a discrete $L^2$ error that
grows superlinearly in time. Reducing the relative tolerance of Newton's
method to $10^{-4}$ reduces the error growth rate to linear, resulting
in significantly better results for long-time simulations. Applying
relaxation after each time step with a relative tolerance of $10^{-3}$
for Newton's method also yields a linear error growth rate and even
slightly smaller discrete $L^2$ errors for long-time simulations.

%%=============================================================%%

\section{Conclusion} \label{sec:conclusion}

We have analyzed entropy properties of iterative solvers
in the context of time integration methods for nonlinear conservation laws. Continuing recent work on linear invariants in
\cite{birken2022conservation,linders2022locally}, we have focused on
the conservation of nonlinear functionals. In particular, we have considered combinations of space and implicit time discretization that result in entropy conservative and entropy dissipative schemes when the arising equation systems were solved exactly. In practice, the iterative solver is terminated once a tolerance is reached. This tolerance is chosen such that the iteration error is smaller than the time integration error. We have demonstrated that, in this situation, Newton's method can result in a qualitatively wrong behavior of the entropy, both
for entropy conservative and dissipative schemes.

Based on an analysis for Burgers' equation, we have explored several possible entropy fixes for Newton's method. Of these, an idea stemming from the recently developed relaxation methods, is most performant. These methods are designed as small modifications of time integrations schemes that are able to preserve the correct evolution of nonlinear functionals. Here we have shown that as long as the iteration error is small enough in the sense described above, they can also be used within implicit time integrators with inexact solves. We have demonstrated that Newton's method with inexact linear solves and reasonable tolerances combines well with the relaxation approach, in particular for nonlinear dispersive wave equations. The numerical results show that, for the problems considered here, entropy conservation leads to smaller errors than non-conservative methods, even when the tolerance of the iterative method is an order of magnitude larger.

%%=============================================================%%

\section*{Acknowledgments}

We thank Gregor Gassner for stimulating discussions about this
research topic and comments on an early draft of the manuscript.

%%=============================================================%%

\section*{Statements and Declarations}

\subsection*{Funding}

Viktor Linders was partially funded by The Royal Physiographic Society in Lund. \\

\noindent Hendrik Ranocha was supported by the Deutsche Forschungsgemeinschaft
(DFG, German Research Foundation, project number 513301895)
and the Daimler und Benz Stiftung (Daimler and Benz foundation,
project number 32-10/22).

\subsection*{Code availability }

We have set up a reproducibility repository \cite{linders2023resolvingRepro}
for this article, containing all Julia source code required to fully
reproduce the numerical experiments discussed in this article.

%========================================================================

\appendix

\section{Entropy analysis for Lobatto IIIC} \label{app:Lobatto}
The purpose of this appendix is to provide details of the entropic behavior of the Lobatto IIIC method used in the numerical experiments in Section \ref{sec:burgers}. The analysis will be kept general enough to encompass methods of arbitrary order of accuracy. It utilizes the Summation-By-Parts (SBP) property \cite{nordstrom2013summation,boom2015high} satisfied by the RK method. The same argument can be made for any RK method associated with the SBP framework such as Radau IA and IIA \cite{ranocha2019some}. See \cite{linders2020properties,versbach2022theoretical} for details about the properties and implementation of such methods.

\subsection{The fully discrete scheme}

For a system of ordinary differential equations $\xvec{u}_t = \xvec{f}(\xvec{u})$, the Runge-Kutta stage equations and update are given by
\begin{equation} \label{eq:RK_stage_and_update}
\begin{aligned}
  \xvec{y}^i
  &=
  \xvec{u}^{n} + \dt_n \sum_{j=1}^{s} a_{ij} \, \xvec{f}\bigl(\xvec{y}^j\bigr),
  \qquad i \in \{1, \dots, s\},
  \\
  \xvec{u}^{n+1}
  &=
  \xvec{u}^{n} + \dt_n \sum_{i=1}^{s} b_{i} \, \xvec{f}\bigl(\xvec{y}^i\bigr).
\end{aligned}
\end{equation}
In the experiment in Section \ref{sec:burgers}, the 3-stage Lobatto IIIC method is used, which is associated with the Butcher tableau
\[
\renewcommand\arraystretch{1.2}
\begin{array}
{c|ccc}
0 & 1/6 & -1/3 & \phantom{-}1/6 \\
1/2 & 1/6 & \phantom{-}5/12 & -1/12 \\
1 & 1/6 & \phantom{-}2/3 & \phantom{-}1/6 \\
\hline
& 1/6 & \phantom{-}2/3 & \phantom{-}1/6
\end{array}.
\]
It will be convenient for our purposes to express \eqref{eq:RK_stage_and_update} in a vector format as
\begin{equation} \label{eq:RK_stages}
\begin{aligned}
\xvec{F}(\xvec{U}) :&= \xvec{U} - \vec{1} \kron \xvec{u}^n + \dt_n (\vec{A} \kron \xvec{I}) \xvec{f}(\xvec{U}) = \xvec{0}, \\
\xvec{u}^{n+1} &= \xvec{y}^s.
\end{aligned}
\end{equation}
Here, $\vec{A} = \{ a_{ij} \}_{i,j=1}^s$ is the Butcher coefficient matrix, $\vec{1}$ is the $s$-element vector of all ones, $\xvec{U}$ is the (column) vector containing the stacked stage vectors $\xvec{y}^1, \dots, \xvec{y}^s$ and $\kron$ denotes the Kronecker product.

The update $\xvec{u}^{n+1}$ can be computed as in \eqref{eq:RK_stages} due to the fact that the final row of $\vec{A}$ is identical to the vector $\vec{b}^\top$. This property can be generalized by considering methods for which there is a vector $\xvec{v} \in \mathbb{R}^s$ such that $\vec{A}^\top \vec{v} = \vec{b}$, in which case $\xvec{u}^{n+1} = (\vec{v}^\top \kron \xvec{I}) \xvec{U}$. For the 3-stage Lobatto IIIC method, we have $\vec{v} = (0,0,1)^\top$.

We now revisit the spatial discretization of Burgers' equation in \eqref{eq:Burgers_semidiscrete} by setting $\xvec{f}(\xvec{u}) = -2(\xvec{D} \diag{\xvec{u}} \xvec{u} + \diag{\xvec{u}} \xvec{D} \xvec{u})$. The entropy behavior of the fully discrete scheme can be analyzed with the aid of the SBP property, which can be expressed as
\begin{equation} \label{eq:SBP_property}
\vec{B} \vec{A}^{-1} + \vec{A}^{-\top} \vec{B} = \diag{1,0,\dots,0,1} = \vec{e}_1 \vec{e}_1^\top + \vec{e}_s \vec{e}_s^\top.
\end{equation}
Here, $\vec{B} = \diag{\vec{b}}$ and $\vec{e}_j$ denotes the $j$th column of the $s \times s$ identity matrix. In this context, $\vec{A}^{-1}$ can be viewed as a difference operator adjoined with an initial condition, and $\vec{B}$ as a quadrature rule \cite{linders2018order}. The SBP property \eqref{eq:SBP_property} is thus a discrete version of integration by parts.

We begin by left-multiplying the stage equations in \eqref{eq:RK_stages} by $\dx \xvec{U}^\top (\vec{B} \vec{A}^{-1} \kron \xvec{I})$. This is a well-defined operation since $\vec{A}$ is invertible for any $s$ \cite{linders2022eigenvalue}. By a derivation identical to \eqref{eq:entropy_midpoint_derivation}, it holds that $(\xvec{y}^i)^\top \xvec{f}(\xvec{y}^i) = \xvec{0}$ for each $i=1,\dots,s$. Since $\vec{B}$ is diagonal it follows that $\dx \xvec{U}^\top (\vec{B} \kron \xvec{I}) \xvec{f}(\xvec{U}) = \xvec{0}$, and consequently
\begin{equation} \label{eq:RK_stage_entropy}
\dx \xvec{U}^\top (\vec{B} \vec{A}^{-1} \kron \xvec{I}) \xvec{U} = \dx \xvec{U}^\top (\vec{B} \vec{A}^{-1} \vec{1} \kron \xvec{u}^n).
\end{equation}
The left-hand side of \eqref{eq:RK_stage_entropy} is a quadratic form and is therefore equal to its symmetric part. The right-hand side is simplified by the relation $\vec{B} \vec{A}^{-1} \vec{1} = \vec{e}_1$; see \cite[Lemma 3]{linders2020properties}. The identity \eqref{eq:RK_stage_entropy} therefore reduces to
\[
\dx \xvec{U}^\top \left( \frac{\vec{B} \vec{A}^{-1} + \vec{A}^{-\top} \vec{B}}{2} \kron \xvec{I} \right) \xvec{U} = \dx \xvec{U}^\top (\vec{e}_1 \kron \xvec{u}^n).
\]
Simplification using the SBP property \eqref{eq:SBP_property} allows us to express this in terms of the first and last stages as
\[
\frac{1}{2} \| \xvec{y}^1 \|^2 + \frac{1}{2} \| \xvec{y}^s \|^2 = \dx (\xvec{y}^1)^\top \xvec{u}^n.
\]
By adding and subtracting $\frac{1}{2} \| \xvec{u}^n \|^2$ from the right-hand side, the entropy, given by $\eta(\xvec{u}^{n+1}) = \frac{1}{2} \| \xvec{y}^s \|^2$, can be expressed as
\[
\begin{aligned}
\eta(\xvec{u}^{n+1}) &= \frac{1}{2} \| \xvec{u}^n \|^2 - \frac{1}{2} \| \xvec{y}^1 \|^2 + \dx (\xvec{y}^1)^\top \xvec{u}^n - \frac{1}{2} \| \xvec{u}^n \|^2 \\
&= \eta(\xvec{u}^n) - \eta(\xvec{y}^1 - \xvec{u}^n).
\end{aligned}
\]
Lobatto IIIC therefore dissipates entropy. This result is independent of the number of stages and holds more generally for RK methods with the SBP property. A slight generalization (to the so called gSBP property \cite{fernandez2014generalized}) is necessary for the analysis to hold for some RK methods such as the Radau family.

\subsection{Newton's method}

The Jacobian $\xvec{F}'$ is explicitly given by
\begin{equation} \label{eq:Lobatto_Jacobian}
\xvec{F}'({\xit{U}{k}}) = \xvec{I} + \dt_n (\vec{A} \kron \xvec{I}) \xvec{f}'(\xit{U}{k}),
\end{equation}
where $\xvec{f}'$ is the Jacobian of the spatial discretization. The form of $\xvec{f}'$ is identical to that seen for the midpoint rule, except that it is now repeated for each stage $\xvec{y}^i$ in a block-diagonal fashion. This leads to an equation for $\xit{U}{k+1}$ of the form
\begin{equation} \label{eq:Lobatto_Newton}
\xit{U}{k+1} - \vec{1} \kron \xvec{u}^n + \dt_n (\vec{A} \kron \xvec{I}) \xvec{f}(\xit{U}{k}) + 2 \dt_n (\vec{A} \kron \xvec{I}) \tilde{\xvec{M}} \Delta \xvec{U} = \xvec{0}.
\end{equation}
The matrix $\tilde{\xvec{M}}$ is block-diagonal with each block identical to the case for the midpoint rule, but evaluated at the individual stages. As before, \eqref{eq:Lobatto_Newton} represents a linearization of the fully discrete scheme around the iterate $\xit{U}{k}$, perturbed by a term arising from the spatial Jacobian.

By an analysis completely analogous to that of in the previous subsection, the entropy relation evaluates to
\[
\eta(\xvec{u}^{n+1}) = \eta(\xvec{u}^n) - \eta(\xvec{y}^1 - \xvec{u}^n) - 2 \dt_n \dx \Delta \xvec{U}^\top \xvec{M} \Delta \xvec{U},
\]
where $\xvec{M} = (\vec{B} \kron \xvec{D}) \diag{\xit{U}{k}} + \diag{(\vec{B} \kron \xvec{D}) \xit{U}{k}}$. The entropy error induced by Newton's method is thus of the same form as for the midpoint rule (for which $\vec{B}=1$), but includes all $s$ stages of the Runge-Kutta scheme.

%========================================================================

\bibliographystyle{hsiam}
\bibliography{references_arxiv.bib}

%========================================================================

\end{document}